\documentclass[reqno]{amsart}
\usepackage{amsmath}
\usepackage{amsthm, amscd, amssymb, amsfonts, amsbsy}
\usepackage[usenames, dvipsnames]{color}
\usepackage{enumerate}
\usepackage{hyperref}
\usepackage{mathrsfs}
\usepackage{verbatim}
\usepackage{todonotes}

\numberwithin{equation}{section}

\theoremstyle{plain}
\newtheorem{theorem}{Theorem}[section]
\newtheorem{lemma}[theorem]{Lemma}

\theoremstyle{definition}
\newtheorem{definition}[theorem]{Definition}
\newtheorem{assumption}[theorem]{Assumption}

\theoremstyle{remark}
\newtheorem{remark}[theorem]{Remark}
\newtheorem{notation}{Notation}[section]

\makeatletter
\def\dashint{\operatorname%
{\,\,\text{\bf--}\kern-.98em\DOTSI\intop\ilimits@\!\!}}
\makeatother

\def\bR{\mathbb{R}}

\def\cB{\mathcal{B}}

\def\cL{\mathcal{L}}

\def\cR{\mathcal{R}}

\begin{document}
\title[Stokes system]{Green functions for stationary Stokes systems with conormal derivative boundary condition in two dimensions}

\author[J. Choi]{Jongkeun Choi}
\address[J. Choi]{
Department of Mathematics Education, Pusan National University,  Busan, 46241, Republic of Korea}

\email{jongkeun\_choi@pusan.ac.kr}

\thanks{J. Choi was supported by the National Research Foundation of Korea (NRF) under agreement NRF-2022R1F1A1074461}

\author[D. Kim]{Doyoon Kim}
\address[D. Kim]{Department of Mathematics, Korea University, 145 Anam-ro, Seongbuk-gu, Seoul, 02841, Republic of Korea}
\email{doyoon\_kim@korea.ac.kr}

\thanks{D. Kim was supported by the National Research Foundation of Korea (NRF) funded by the Korea government (MSIT) (2019R1A2C1084683).}


\begin{abstract}
We construct Green functions of conormal derivative problems for the stationary Stokes system with measurable coefficients in a two dimensional Reifenberg flat domain. 
\end{abstract}

\maketitle

\section{Introduction}	

We present Green functions for the stationary Stokes system with measurable coefficients and the  {\em{conormal}} derivative boundary condition in two dimensional domains, along with a unified approach for constructing Green functions for both elliptic and Stokes systems with various boundary conditions.

Let $\Omega$ be a bounded domain in $\bR^2$.
We consider an elliptic operator $\cL$ and its associated conormal derivative operator $\cB$ of the form
$$
\cL u=D_\alpha (A_{\alpha\beta}D_\beta u), \quad \cB u=\nu^\alpha A_{\alpha\beta} D_\beta  u
$$
acting on column vector-valued functions  $u=(u^1, u^2)^{\top}$, where $\nu=(\nu^1, \nu^2)^\top$ is the outward unit normal to $\partial \Omega$.
We use the Einstein summation convention on repeated indices.
By a Green function of the operator $\cL$ in $\Omega$, we mean a pair $(G, \Pi)=(G(x,y), \Pi(x,y))$ satisfying 
$$
\begin{cases}
\operatorname{div} G(\cdot,y)=0 &\quad \text{in }\, \Omega,\\
\cL G(\cdot,y)+\nabla \Pi(\cdot,y)=-\delta_yI+\frac{1}{|\Omega|}I&\quad \text{in }\, \Omega,\\
\cB G(\cdot, y)+\nu\Pi(\cdot,y) =0& \quad \text{on }\, \partial \Omega.
\end{cases}
$$
See Section \ref{sect_2_2} for a precise definition of the Green function.
As is well known, 
 if  the Green function $(G, \Pi)$ of $\cL$ exists and $(u, p)$ is a weak solution of the adjoint problem
\begin{equation}		\label{210827@eq1}
\begin{cases}
\operatorname{div} u=g &\quad \text{in }\, \Omega,\\
\cL^\star u+\nabla p=D_\alpha f_\alpha &\quad  \text{in }\, \Omega,\\
\cB^\star u+p \nu=\nu^\alpha f_\alpha  & \quad \text{on }\, \partial \Omega
\end{cases}
\end{equation}
with bounded data, then the flow velocity $u$ has the following representation
$$
u(y)=\int_\Omega D_\alpha G(x,y)^{\top} f_\alpha(x)\,dx+\int_\Omega \Pi(x,y)^\top g(x)\,dx.
$$
The conormal problem  \eqref{210827@eq1}  arises from the variational principle and have been used, for example, to describe a channel flow model in which the output velocity dependence is a prior unknown; see  \cite{MR3458302, MR1972357} and references therein.

In this paper, we prove that the  Green function $(G, \Pi)$ of $\cL$ exists and has the logarithmic pointwise bound 
$$
|G(x,y)|\le N\log \bigg(\frac{\operatorname{diam}(\Omega)}{|x-y|}\bigg)+N
$$
if $\Omega$ is a bounded Reifenberg flat domain in $\bR^2$; see Theorem \ref{M1}.
Note that Lipschitz domains with small Lipschitz constants are Reifenberg flat.
Thus, our main theorem holds for all Lipschitz domains with small Lipschitz constants.
In fact, in this case, since our argument is irrelevant to the flatness of the Lipschitz boundary, 
the theorem still holds for all Lipschitz domains with a bounded Lipschitz constant; see Remark \ref{rem0712_1}.
The Green function also satisfies the  pointwise bound
$$
|D_xG(x,y)|+|\Pi(x,y)| \le \frac{N}{|x-y|}
$$
away from $\partial \Omega$ if we further assume  that the coefficients are in the class of DMO (Dini mean oscillations); see Remark \ref{210819@rmk2}.
As far as the existence of the Green function is concerned, the coefficients $A_{\alpha\beta}$ need only be measurable.
Stokes systems with such  irregular coefficients can be used to describe the motion of inhomogeneous fluids with density-dependent viscosity and Stokes flow over composite materials with closely spaced interfacial boundaries. 
See, for instance,  \cite{MR2663713,MR4440019,  MR3758532, MR0425391, MR1422251, MN1987,MR3078336}.

The current paper can be viewed as a continuation of \cite{MR3877495}, where the corresponding three or higher dimensional Green functions are studied.
With  Dong, we proved in \cite{MR3877495} that when the dimension $d\ge 3$, the Green function $(G, \Pi)$ of $\cL$ exists and has pointwise bound 
$$
|G(x,y)|\le N|x-y|^{2-d}
$$
 if coefficients are in the class of BMO (bounded mean oscillations).
See \cite{MR3693868, MR3670039} on the corresponding results for Green functions of Dirichlet problems and fundamental solutions.
The proof in \cite{MR3877495} as well as \cite{MR3693868, MR3670039} relies on utilizing $W^1_2$-solvability with the Sobolev embedding 
\begin{equation}		\label{220903@eq1}
W^1_2\hookrightarrow L_{2d/(d-2)}
\end{equation}
to construct a family of approximated Green functions, establishing uniform estimates for the family, and applying a compactness theorem.
This approach was  given by Gr\"uter-Widman \cite{MR0657523} for Dirichlet Green functions for elliptic equations with measurable coefficients and further developed in  \cite{MR2341783, MR2718661} for those for elliptic systems with coefficients in the class of VMO (vanishing mean oscillations).
See also \cite{MR3105752} for Neumann Green functions for the elliptic systems. 
Remark that in all these work, the dimension $d$ is assumed to be greater than or equal to $3$ because of the usage of the embedding \eqref{220903@eq1}.

To prove the main result of this paper, we refine the aforementioned approach to be applicable to the setting for $\bR^2$, which is precisely where the strength of our paper lies.
The key ingredient is to utilize $W^{1}_r$-solvability with $r$ to be closed to $2$ combined with the Sobolev embedding $W^1_q\hookrightarrow L_{2q/(2-q)}$ for $q\in (1,2)$.
Note that the $W^{1}_r$-solvability follows from the standard reverse H\"older technique, which in general works for  differential operators with measurable coefficients.
This allows us to avoid the lack of the embedding \eqref{220903@eq1} when $d=2$ while keeping no regularity assumptions on the coefficients.
We also note that   our approach is general enough to permit the domain to be Reifenberg flat or Lipschitz, and allow other boundary value problems, such as mixed Dirichlet-conormal and Robin problems for both elliptic and Stokes systems.  
Therefore, once the corresponding solvability and embedding results are achieved for a boundary value problem (for instance, see \cite{MR4261267} for the reverse H\"older technique and embedding results related to the mixed problem),  we believe that, based on our approach, one can derive the corresponding two dimensional Green function for the problem with measurable coefficients in the domain.

Finally, we remark on some
 approaches to constructing two dimensional Green functions.
Regarding Dirichlet problems, 
we refer the reader to Dolzmann-M\"uller \cite{MR1354111}, where the authors gave an approach for the existence of the Green functions having logarithmic pointwise bounds for elliptic systems with measurable coefficients.
Their argument is based on  the solvability in weak Lebesgue spaces applied to the equation
$$
-D_i(a_{ij} D_j u)=\operatorname{div} F,\quad F(x)=\frac{|x|^{\top}}{2\pi |x|^2},
$$
where the right-hand side of the equation is indeed the Dirac delta function. 
This argument was adapted in  \cite{MR3906316} for Green functions for the Stokes systems with measurable coefficients.
The argument also works in higher dimensions, but it is questionable to be applicable to other boundary value problems.
We also refer the reader to D. Mitrea-I. Mitrea \cite{MR2763343} for Dirichlet Green functions with weak type estimates in Lipschitz domains.
Their proof relies on the theory of operators on interpolation scale of spaces associated with  second and higher order systems with constant coefficients, the bi-Laplacian, and the classical Stokes system.
For another approach, we refer to Dong-Kim \cite{MR2485428}, where they constructed Green functions for the elliptic systems by using heat kernel estimates, the argument in which requires first establishing pointwise bounds for the heat kernels.
See also \cite{MR2886465, MR3017032, MR3261109} and the references therein for work in this direction.
Lastly, we would like to mention two papers \cite{MR3169756, MR3320459} on the Green functions for the mixed problems in two dimensions.
In \cite{MR3169756} Talyor et al. presented another adaptation of the approach in \cite{MR0657523} for the Green functions of the mixed or Neumann problems for elliptic systems in a Lipschitz domain.
For this, the authors used an embedding of the Sobolev spaces into the spaces of BMO instead of \eqref{220903@eq1} and a trace theorem on the Lipschitz domain.  
In \cite{MR3320459}, Ott et al. established the existence of the Green functions for the classical Stokes system by considering  the solvability of the system with data lying in dual of  Lorentz-Sobolev spaces.

The remainder of this paper is  organized as follows.
In Section \ref{S2}, we state our main results along with some notation, assumptions, and the definition of the Green function.
We provide some auxiliary results in Section \ref{S3}.
Finally, we prove the main theorem in Sections \ref{S4} and \ref{S5}.
In the Appendix, we prove in detail a theorem for a chain of balls in a Reifenberg flat domain, which is needed to prove the logarithmic bound of the Green function.

\section{Preliminaries and main result}	\label{S2}

We first introduce basic notation  used throughout the paper.
Let  $\Omega$ be a bounded domain in $\bR^2$.
For a function $u$ in $\Omega$, we set 
$$
\|u\|_{L_{2, \infty}(\Omega)}=\sup_{t>0} t\big|\{x\in \Omega:|u(x)|>t\}\big|^{1/2}.
$$
For $1\le q\le \infty$, we define
$$
\tilde{L}_q(\Omega)=\{u\in L_q(\Omega):(u)_\Omega=0\}, \quad \widetilde{W}^{1}_{q}(\Omega)=\{u\in W^1_q(\Omega):(u)_\Omega=0\},
$$
where  $W^1_q(\Omega)$ is the usual Sobolev space and 
$$
(u)_\Omega=\dashint_{\Omega} u\,dx=\frac{1}{|\Omega|}\int_\Omega u\,dx.
$$
We  denote
$$
\Omega_r(x)=\Omega\cap B_r(x),
$$
where $B_r(x)$ is an Euclidean disk of radius $r$ centered at $x$.
By, for instance, $f \in L_q(\Omega)^2$ we mean that $f = (f^1,f^2)^\top$ and $f^1,f^2 \in L_q(\Omega)$.
As a superscript we also use $2\times 2$ (resp. $1\times 2$) in place of $2$ to denote a space for $2\times 2$ (resp. $1\times 2$) matrix valued functions.

\subsection{Conormal derivative problem}	

Let $\cL$ be a differential operator in divergence form
$$
\cL u=D_\alpha(A_{\alpha\beta}D_\beta u),
$$
where the coefficients $A_{\alpha\beta}=A_{\alpha\beta}(x)$ are $2\times 2$ matrix-valued functions in $\bR^2$ satisfying the strong ellipticity condition, that is, there is a constant $\lambda\in (0, 1]$ such that for any $x\in \bR^2$ and $\xi_\alpha\in \bR^2$, $\alpha\in \{1,2\}$, we have 
\begin{equation}		\label{210817@eq1}
|A_{\alpha\beta}(x)|\le \lambda^{-1}, \quad \sum_{\alpha,\beta=1}^2 A_{\alpha\beta}(x)\xi_\beta\cdot \xi_\alpha\ge \lambda\sum_{\alpha=1}^2 |\xi_\alpha|^2.
\end{equation}
Let $\cB$ be the conormal derivative operator associated with $\cL$ given by 
$$
\cB u=\nu^\alpha A_{\alpha\beta}D_\beta u,
$$
where $\nu=(\nu^1, \nu^2)^\top$ is the outward unit normal to $\partial \Omega$.
We define the adjoint operator $\cL^\star$ and its associated conormal derivative operator $\cB^\star$  by 
$$
\cL^\star u=D_\alpha((A_{\beta\alpha})^{\top} D_\beta u), \quad \cB^\star u= \nu^\alpha (A_{\beta\alpha})^{\top} D_\beta u.
$$
Note that the coefficients of  $\cL^\star$ also satisfy the ellipticity condition \eqref{210817@eq1} with the same constant $\lambda$.

Let $\Omega$ be a bounded domain in $\bR^2$ and $q,q_1\in (1,\infty)$ with $q_1\ge 2q/(2+q)$.
For $f\in \tilde{L}_{q_1}(\Omega)^2$, $f_\alpha\in L_q(\Omega)^2$,  and $g\in L_q(\Omega)$, we say that $(u,p)\in W^1_
q(\Omega)^2\times L_q(\Omega)$ is a weak solution of the problem
$$
\begin{cases}
\operatorname{div}u=g &\quad \text{in }\, \Omega,\\
\cL u+\nabla p=f+D_\alpha f_\alpha &\quad \text{in }\, \Omega,\\
\cB u+p\nu= \nu^\alpha f_\alpha &\quad \text{on }\, \partial \Omega,
\end{cases}
$$
if $\operatorname{div} u=g$ a.e. in $\Omega$ and 
$$
\int_\Omega A_{\alpha\beta}D_\beta u\cdot D_\alpha \phi\,dx+\int_\Omega p\operatorname{div}\phi\,dx=-\int_\Omega f\cdot \phi\,dx+\int_\Omega f_\alpha \cdot D_\alpha \phi\,dx
$$
holds for any $\phi\in W^1_{q/(q-1)}(\Omega)^2$.
Similarly, we  say that $(u, p)\in W^{1}_q(\Omega)^2\times L_q(\Omega)$ is a weak solution of the adjoint problem
$$
\begin{cases}
\operatorname{div}u=g &\quad \text{in }\, \Omega,\\
\cL^\star u+\nabla p=f+D_\alpha f_\alpha &\quad \text{in }\, \Omega,\\
\cB^\star u+p\nu=\nu^\alpha f_\alpha  &\quad \text{on }\, \partial \Omega,
\end{cases}
$$
if $\operatorname{div} u=g$ a.e. in $\Omega$ and 
$$
\int_\Omega A_{\alpha\beta}D_\beta \phi\cdot D_\alpha u\,dx+\int_\Omega p\operatorname{div}\phi\,dx=-\int_\Omega f\cdot \phi\,dx+\int_\Omega f_\alpha \cdot D_\alpha \phi\,dx
$$
holds for any $\phi\in W^1_{q/(q-1)}(\Omega)^2$. 

We remark that even when $\Omega$ is irregular so that neither the outer normal nor the trace of a $W^1_q(\Omega)$ function on the boundary is defined, the weak formulations above make sense because no boundary terms appear there.
 
\subsection{Green function}	\label{sect_2_2}
In the definition below, $G=G(x,y)$ is a $2\times 2$ matrix-valued function, $\Pi=\Pi(x,y)$ is a $1\times 2$ vector-valued function, $I$ is the $2\times 2$ identity matrix,  and $\delta_y$ is the Dirac delta function concentrated at $y$.

\begin{definition}		\label{D1}
Let $\Omega$ be a bounded domain in $\bR^2$.
We say that $(G, \Pi)$ is a Green function (for the flow velocity) of $\cL$ in $\Omega$ if it satisfies the following properties.
\begin{enumerate}[$(i)$]
\item
For any $y\in \Omega$, 
$$
G(\cdot,y)\in \widetilde{W}^1_1(\Omega)^{2\times 2}, \quad 
\Pi(\cdot,y)\in  L_1(\Omega)^{1\times 2}.
$$
\item
For any $y\in \Omega$, $(G(\cdot,y), \Pi(\cdot,y))$ satisfies 
\begin{equation}		\label{210815@A1}
\begin{cases}
\operatorname{div} G(\cdot,y)=0 &\quad \text{in }\, \Omega,\\
\cL G(\cdot,y)+\nabla \Pi(\cdot,y)=-\delta_yI+\frac{1}{|\Omega|}I&\quad \text{in }\, \Omega,\\
\cB G(\cdot, y)+\nu\Pi(\cdot,y) =0& \quad \text{on }\, \partial \Omega,
\end{cases}
\end{equation}
in the sense that 
for any $k\in \{1,2\}$ and $\phi\in \widetilde{W}^1_\infty(\Omega)^2\cap C(\Omega)^2$, we have 
$$
\operatorname{div} G^{\cdot k}(\cdot,y)=0 \quad \text{a.e. in }\, \Omega
$$
and 
$$
\begin{aligned}
&\int_\Omega A_{\alpha\beta} D_\beta G^{\cdot k}(\cdot ,y)\cdot D_\alpha \phi\,dx+\int_\Omega \Pi^k(\cdot ,y)\operatorname{div}\phi \,dx=\phi^k(y),
\end{aligned}
$$
where $G^{\cdot k}(\cdot,y)$ is the $k$th column of $G(\cdot,y)$.
\item
If $(u,p)\in \widetilde{W}^1_2(\Omega)^2\times L_2(\Omega)$ is a weak solution of the adjoint problem 
\begin{equation}		\label{171010@eq1}
\begin{cases}
\operatorname{div} u=g &\quad \text{in }\, \Omega,\\
\cL^\star u+\nabla p=f+D_\alpha f_\alpha &\quad  \text{in }\, \Omega,\\
\cB^\star u+p \nu=\nu^\alpha f_\alpha  & \quad \text{on }\, \partial \Omega,
\end{cases}
\end{equation}
where  $f\in \tilde{L}_{\infty}(\Omega)^2$, $f_\alpha\in L_\infty(\Omega)^2$, and $g\in {L}_\infty(\Omega)$, 
then for a.e. $y\in \Omega$, we have 
$$
u(y)=-\int_\Omega G(\cdot,y)^{\top}f\,dx+\int_\Omega D_\alpha G(\cdot,y)^{\top} f_\alpha\,dx+\int_\Omega \Pi(\cdot,y)^\top g\,dx,
$$
where $G(\cdot,y)^{\top}$ and $\Pi(\cdot,y)^{\top}$ are the transposes of  $G(\cdot,y)$ and $\Pi(\cdot,y)$.
\end{enumerate}
The Green function of the adjoint operator $\cL^\star$ is defined similarly.
\end{definition}

We remark that the property $(iii)$ in Definition \ref{D1} together with the unique solvability of the conormal derivative problem in $\widetilde{W}^1_2(\Omega)^2\times L_2(\Omega)$ (see Lemma \ref{200229@lem1}) gives the uniqueness of  a Green function.

\subsection{Main result}	
Our main result is the existence of the Green function satisfying the logarithmic pointwise bound.
For this, we impose the following Reifenberg flat condition on the boundary of the domain.

\begin{assumption}		\label{A1} 
Let $\gamma\in [0, 1/96]$.
There exists $R_0\in (0,1]$ such that the following holds:
for any $x_0\in \partial \Omega$ and $R\in (0,R_0]$, there is a coordinate system depending on $x_0$ and $R$ such that in this coordinate system (called the coordinate system associated with $(x_0,R)$), we have
$$
\{y:x_{0}^1+\gamma R<y^1\}\cap B_R(x_0)\subset \Omega_R(x_0)\subset\{y:x_{0}^1-\gamma R<y^1\}\cap B_R(x_0),
$$
where $x_{0}^1$ is the first coordinate of $x_0$ in the coordinate system.
\end{assumption}

\begin{remark}
Regarding the flatness parameter $\gamma$ in Assumption \ref{A1}, note that our main theorem (Theorem \ref{M1}) is independent of the size of $\gamma$ as long as $\gamma \leq 1/96$.
Thus, one can just set $\gamma = 1/96$ in  Assumption \ref{A1} insteand of $\gamma \in [0,1/96]$.
Clearly, a boundary satisfying Assumption \ref{A1} also satisfies Assumption \ref{A1} with $\gamma = 1/96$.
Also see Remark \ref{rem0712_1} below.
\end{remark}

In the theorem below, the coefficients $A_{\alpha\beta}$ are assumed to be measurable and satisfy the ellipticity condition \eqref{210817@eq1}.

\begin{theorem}		\label{M1}
Let $\Omega$ be a bounded domain in $\bR^2$ with $\operatorname{diam}\Omega\le K$ satisfying Assumption \ref{A1}.
Then there exist Green functions $(G, \Pi)$ of $\cL$ and $(G^{\star}, \Pi^\star)$ of $\cL^\star$.
Moreover, $G$ and $G^\star$ are continuous in $\{(x,y)\in \Omega\times \Omega:x\neq y\}$ and satisfy
\begin{equation}		\label{210720@eq2}
G(x,y)=G^\star(y,x)^{\top} \quad \text{for all }\, x,y\in \Omega, \quad x\neq y.
\end{equation}
Furthermore, the following estimates hold.
\begin{enumerate}[$(a)$]
\item
For any $x,y\in \Omega$ and $R\in (0, R_0]$, we have 
\begin{equation}		\label{230208_eq1}
\|DG(\cdot, y)\|_{L_q(\Omega_R(x))}+\|\Pi(\cdot, y)\|_{L_q(\Omega_R(x))}\le N_q R^{-1+2/q},
\end{equation}
where $1\le q<2$.
\item
There exists $q_0=q_0(\lambda)>2$ such that for any $y\in \Omega$ and $R\in (0, R_0]$, we have 
\begin{equation}		\label{210803@eq1}
\|DG(\cdot,y)\|_{L_{q}(\Omega\setminus \overline{B_R(y)})}+\|\Pi(\cdot,y)\|_{L_{q}(\Omega\setminus \overline{B_R(y)})}\le N_q R^{-\mu},
\end{equation}
where $2<q\le q_0$ and $\mu=1-2/q$.
Moreover, for any $x\in \overline{\Omega}$ satisfying $|x-y|\ge R$, we have 
\begin{equation}		\label{210721@eq1a}
[G(\cdot,y)]_{C^{\mu}(\Omega_{R/16}(x))}\le N_q R^{-\mu}.
\end{equation}
\item
For any $y\in \Omega$, we have 
\begin{equation}		\label{210721@eq2}
\|DG(\cdot,y)\|_{L_{2, \infty}(\Omega)}+\|\Pi(\cdot,y)\|_{L_{2,\infty}(\Omega)}\le N.
\end{equation}
\item
For any $x,y\in \Omega$ with $x\neq y$, we have 
\begin{equation}		\label{210721@eq2a}
|G(x,y)|\le N\log \bigg(\frac{K}{|x-y|}\bigg)+N.
\end{equation}
\end{enumerate}
In the above, $N=N(\lambda, R_0, K)>0$ and $N_q$ depends also on $q$.
\end{theorem}

Related to the theorem above, we have a few remarks.

\begin{remark}
Let  $(u,p)\in \widetilde{W}^1_2(\Omega)^2\times L_2(\Omega)$ be a weak solution of 
$$
\begin{cases}
\operatorname{div} u=0 &\quad \text{in }\, \Omega,\\
\cL u+\nabla p=f+D_\alpha f_\alpha &\quad  \text{in }\, \Omega,\\
\cB u+p \nu=\nu^\alpha f_\alpha  & \quad \text{on }\, \partial \Omega,
\end{cases}
$$
where  $f\in \tilde{L}_{\infty}(\Omega)^2$ and $f_\alpha\in L_\infty(\Omega)^2$.
Then, using the counterpart of $(iii)$ in Definition \ref{D1} for $(G^\star, \Pi^\star)$, we have that for a.e. $x\in \Omega$, 
$$
u(x)=-\int_\Omega G^\star(y,x)^\top f(y)\,dy+\int_\Omega D_\alpha G^\star(y, x)^\top f_\alpha(y)\,dy.
$$
Together with \eqref{210720@eq2}, this gives
$$
u(x)=-\int_\Omega G(x,y) f(y)\,dy+\int_\Omega D_\alpha G(x, y) f_\alpha(y)\,dy.
$$
\end{remark}

\begin{remark}		\label{210819@rmk2}
In Theorem \ref{M1}, one may consider $L_\infty$ or pointwise estimate for $DG$ and $\Pi$ under a regularity assumption on the coefficients $A_{\alpha\beta}$.
Indeed, by following the same steps
as in the proof of  \cite[Theorem 3.5]{MR3906316}, we see that 
$$
\operatorname*{ess\,sup}_{B_{|x-y|/4}(x)}\big(|DG(\cdot,y)|+|\Pi(\cdot,y)|\big)\le \frac{N}{|x-y|}
$$
for any $x,y\in \Omega$ with $0<|x-y|\le \frac{1}{2}\operatorname{dist}(y, \partial \Omega)$, provided that $A_{\alpha\beta}$ are of {\em{partial Dini mean oscillation}}.
We shall say that a locally integrable function is of partial Dini mean oscillation if it is merely measurable in one direction and its $L_1$-mean oscillation in the other directions satisfy a Dini condition; see \cite[Definition 2.1]{MR3912724} or \cite[Definition 2.1]{MR3906316} for more precise definition. 
Moreover, as in \cite[Remark 3.6]{MR3906316}, we have 
$$
|D_xG(x,y)|+|\Pi(x,y)|\le \frac{N}{|x-y|}
$$
if we assume further that  $A_{\alpha\beta}$ are of Dini mean oscillation in {\em{all}} the directions.
In the above, the constant $N$ depends only on $\lambda$, $R_0$, $K$, and a Dini function derived from the $L_1$-mean oscillation.
\end{remark}

\begin{remark}
							\label{rem0712_1}
Lipschitz domains with small Lipschitz constants are Reifenberg flat.
Thus, our main theorem holds for all Lipschitz domains provided that their Lipschitz constants are not greater than $1/96$.
On the other hand, the embeddings presented in Section \ref{S3_1} also hold for Lipschitz domains whenever the Lipschitz constants are bounded.
Thus, in our main theorem (Theorem \ref{M1}), instead of Assumption \ref{A1}, one can just assume that the domain is Lipschitz with bounded Lipschitz constant.
That says that the flatness of the boundary is irrelevant to Theorem \ref{M1} for a Lipschitz domain.

When $\Omega$ is a Lipschitz domain so that the boundary trace of a $W^1_q(\Omega)$ function is well defined,
then one may use the normalization condition 
\begin{equation}		\label{210817@eq2}
\int_{\partial \Omega} G(x, y)\,d\sigma_x=0
\end{equation}
instead of 
\begin{equation}		\label{210817@eq2a}
\int_\Omega G(x,y)\,dx=0.
\end{equation}
In this case, the Green function $(G, \Pi)$ can be defined as a solution of the problem 
$$
\begin{cases}
\operatorname{div} G(\cdot,y)=0 &\quad \text{in }\, \Omega,\\
\cL G(\cdot,y)+\nabla \Pi(\cdot,y)=-\delta_yI &\quad \text{in }\, \Omega,\\
\cB G(\cdot, y)+\nu\Pi(\cdot,y) =-\frac{1}{|\partial \Omega|}I& \quad \text{on }\, \partial \Omega.
\end{cases}
$$
We refer the reader to \cite{MR3017032} for Neumann Green functions of elliptic and parabolic systems with the normalization condition \eqref{210817@eq2a} and \cite{MR3105752} for those of elliptic systems with the normalization condition \eqref{210817@eq2}.
\end{remark}

\section{Auxiliary results}	\label{S3}

Hereafter in the paper, we use the following notation.

\begin{notation}
For a given function $f$, if there is a continuous version of $f$, that is, there is a continuous function $\tilde{f}$ such that $\tilde{f}=f$ in the almost everywhere sense, then we replace $f$ with $\tilde{f}$ and denote the version again by $f$.
\end{notation}

\subsection{Embeddings}			\label{S3_1}

In this subsection,  we present some inequalities suitable to our problem.
For $q\in (1,2)$, we denote by $q^\star$ its Sobolev conjugate, i.e., 
$$
q^\star=\frac{2q}{2-q}.
$$
We start with a local Sobolev-Poincar\'e inequality in a Reifenberg flat domain.

\begin{lemma}		\label{210804@lem1}
Let $\Omega$ be a domain in $\bR^2$ satisfying  Assumption \ref{A1} and $u\in W^1_q(\Omega)$ with $q\in (1,2)$.
Then for any $x_0\in \overline{ \Omega}$ and $R\in (0,R_0]$,  we have 
\begin{equation}		\label{210804@eq1}
\|u-(u)_{\Omega_{R/8}(x_0)}\|_{L_{q^\star}(\Omega_{R/8}(x_0))}\le N \|Du\|_{L_q(\Omega_{R}(x_0))},
\end{equation}
where $N=N(q)$.
\end{lemma}

\begin{proof}
If $B_{R/8}(x_0)\subset \Omega$, then \eqref{210804@eq1} follows immediately from the interior Sobolev-Poincar\'e inequality.
Otherwise, i.e., if  $B_{R/8}(x_0)\not\subset \Omega$, 
we take $y_0\in \partial \Omega$ such that $\operatorname{dist}(x_0, \partial \Omega)=|x_0-y_0|$.
It then follows from  \cite[Theorem 3.5]{MR3809039} that 
$$
\|u-(u)_{\Omega_{R/4}(y_0)}\|_{L_{q^\star}(\Omega_{R/4}(y_0))}\le N\|Du\|_{L_q(\Omega_{R/2}(y_0))},
$$
where $N=N(q)$.
Using this, the triangle inequality, and  the chain of inclusions 
\begin{equation}		\label{210706@eq1}
\Omega_{R/8}(x_0)\subset \Omega_{R/4}(y_0)\subset \Omega_{R/2}(y_0)\subset \Omega_{R}(x_0), 
\end{equation}
we obtain that
$$
\begin{aligned}
&\|u-(u)_{\Omega_{R/8}(x_0)}\|_{L_{q^\star}(\Omega_{R/8}(x_0))}\\
&\le \|u-(u)_{\Omega_{R/4}(y_0)}\|_{L_{q^\star}(\Omega_{R/4}(y_0))}+\|(u)_{\Omega_{R/4}(y_0)}-(u)_{\Omega_{R/8}(x_0)}\|_{L_{q^\star}(\Omega_{R/4}(y_0))}\\
&\le N \|u-(u)_{\Omega_{R/4}(y_0)}\|_{L_{q^\star}(\Omega_{R/4}(y_0))}\\
&\le N \|Du\|_{L_q(\Omega_{R}(x_0))},
\end{aligned}
$$
which gives  \eqref{210804@eq1}.
\end{proof}

It is well known that the Sobolev-Poincar\'e inequality holds over $\Omega$ when $\Omega$ is a bounded Reifenberg flat domain since it is in the category of extension domains.
Below, for the reader's convenience, we present this result with precise information on the parameters on which the constant in the inequality depends.

\begin{lemma}		\label{210709@lem2}
Let $\Omega$ be a bounded domain in $\bR^2$ with $\operatorname{diam}\Omega\le K$ satisfying  Assumption \ref{A1} and  $u\in W^1_q(\Omega)$ with $q\in (1,2)$.  
Then we have
$$
\|u-(u)_{\Omega}\|_{L_{q^\star}(\Omega)}\le N\|Du\|_{L_q(\Omega)},
$$
where   $N=N(R_0, K, q)$.
\end{lemma}

\begin{proof}
Note that for any $x,y,z\in \Omega$ satisfying 
$$
|x-y|\le R_0/16, \quad z\in \Omega_{R_0/16}(x)\cap \Omega_{R_0/16}(y),
$$
by the triangle inequality, \eqref{210804@eq1}, and the fact that  
$$
\Omega_{R_0/16}(z)\subset \Omega_{R_0/8}(x)\cap \Omega_{R_0/8}(y),
$$
we have
\begin{align}
\nonumber
&\int_{\Omega_{R_0/8}(y)} \int_{\Omega_{R_0/8}(x)} |u(\tilde{x})-u(\tilde{y})|^{q^\star}\,d\tilde{x}\,d\tilde{y}\\
\nonumber
&\le N \int_{\Omega_{R_0/8}(x)} \big|u-(u)_{\Omega_{R_0/16}(z)}\big|^{q^\star}\,d\tilde{x} + N\int_{\Omega_{R_0/8}(y)} \big|u-(u)_{\Omega_{R_0/16}(z)}\big|^{q^\star}\,d\tilde{y}\\
\nonumber
&\le N\int_{\Omega_{R_0/8}(x)} \big|u-(u)_{\Omega_{R_0/8}(x)}\big|^{q^\star}\,d\tilde{x} +N\int_{\Omega_{R_0/8}(y)} \big|u-(u)_{\Omega_{R_0/8}(y)}\big|^{q^\star}\,d\tilde{y}\\
\label{210721@B1}
&\le N \|Du\|_{L_q(\Omega_{R_0}(x))}^{q^\star}+N\|Du\|^{q^\star}_{L_q(\Omega_{R_0}(y))},
\end{align}
where $N=N(R_0, q)$.

Now, we choose $z_1,\ldots, z_m\in \Omega$, where $m=m(R_0, K)$, satisfying 
$$
|z_i-z_{i+1}|\le R_0/16, \quad i\in \{1,\ldots, m-1\},
$$
and 
$$
\Omega\subset \bigcup_{i=1}^m \Omega_{R_0/8}(z_i).
$$
Then by the triangle inequality and \eqref{210721@B1}, we have 
$$
\begin{aligned}
\int_{\Omega}|u-(u)_\Omega|^{q^\star}\,dx
&\le N \sum_{i,j=1}^m \int_{\Omega_{R_0/8}(z_i)} \int_{\Omega_{R_0/8}(z_j)} |u(x)-u(y)|^{q^\star}\,dx\,dy\\
&\le N \sum_{i=1}^{m-1} \int_{\Omega_{R_0/8}(z_i)} \int_{\Omega_{R_0/8}(z_{i+1})} |u(x)-u(y)|^{q^\star}\,dx\,dy,
\end{aligned}
$$
where $N=N(R_0, K, q)$.
Indeed, to see the second inequality above, for $i+1 < j$, one can use
$$
\begin{aligned}
&\int_{\Omega_{R_0/8}(z_i)}\int_{\Omega_{R_0/8}(z_j)} |u(x) - u(y)|^{q^\star}  \, dx\,dy \\
&\leq N \int_{\Omega_{R_0/8}(z_i)}\int_{\Omega_{R_0/8}(z_{i+1})} |u(x) - u(y)|^{q^\star} \, dx\,dy\\
&\quad + N \int_{\Omega_{R_0/8}(z_{i+1})}\int_{\Omega_{R_0/8}(z_j)} |u(x) - u(y)|^{q^\star} \, dx\,dy,
\end{aligned}
$$
where the second term can be similarly split if $i+2 < j$.
Together with \eqref{210721@B1}, the above inequalities give the desired estimate.
\end{proof}

\begin{lemma}		\label{200308@lem1}
Let $\Omega$ be a domain in $\bR^2$  satisfying  Assumption \ref{A1} and $u\in W^{1}_{q}(\Omega)$ with $q>2$. 
Then for any $x_0\in \overline \Omega$ and $R\in (0,R_0]$, we have 
\begin{equation}		\label{210804@eq3}
[u]_{C^{1-2/q}(\Omega_{R/8}(x_0))}\le N \|Du\|_{L_q(\Omega_{R}(x_0))}
\end{equation}
and 
\begin{equation}		\label{210804@eq3a}
\|u\|_{L_\infty(\Omega_{R/8}(x_0))}\le N R^{1-2/q}\|Du\|_{L_q(\Omega_{R}(x_0))}+NR^{-2}\|u\|_{L_1(\Omega_{R/8}(x_0))},
\end{equation}
where $N=N(q)$.
\end{lemma}

\begin{proof}
We only prove \eqref{210804@eq3} because \eqref{210804@eq3a} is an easy consequence of \eqref{210804@eq3}. 
Due to the interior Morrey’s inequality along with the technique of the chain of inclusions (see \eqref{210706@eq1}), we only need to show 
$$
[u]_{C^{1-2/q}(\Omega_{R/4}(x_0))}\le N \|Du\|_{L_q(\Omega_{R/2}(x_0))},
$$
where $N=N(q)$,
provided that  $x_0\in \partial \Omega$ and $R\in (0, R_0]$.
For this, we claim that 
for any $x,y\in \Omega_{R/4}(x_0)$ with $x\neq y$, 
\begin{equation}		\label{200303@eq1}
\frac{|u(x)-u(y)|}{|x-y|^{1-2/q}}\le N \|Du\|_{L_{q}(\Omega_{R/2}(x_0))}.
\end{equation}
Set $r=|x-y|$.
We consider the following two cases:
$$
r<R/32, \quad r\ge R/32.
$$
\begin{enumerate}[i.]
\item
$r<R/32$.
If $r< \operatorname{dist}(x,\partial \Omega)$, i.e., $B_{r}(x)\subset\Omega$, then \eqref{200303@eq1} follows from Morrey's inequality in the ball $B_{r}(x)$ with the fact that $B_{r}(x)\subset \Omega_{R/2}(x_0)$.
If  $r\ge \operatorname{dist}(x,\partial \Omega)$, we take $\hat{x}\in \partial \Omega$ such that $|x-\hat{x}|=\operatorname{dist}(x,\partial \Omega)$. 
Then,
\[
x,y \in \Omega_{3r}(\hat{x}) \subset \Omega_{6r}(\hat{x})\subset \Omega_{R/2}(x_0).
\]
By the last inequality in  \cite[p. 2372]{MR3809039} and H\"older's inequality, there exists a constant $\bar{u}$ such that
$$
|u(z)-\bar{u}|
\le N \int_{\Omega_{6r}(\hat{x})} \frac{|Du(\xi)|}{|z-\xi|}\,d\xi\le N r^{1-2/q}\|Du\|_{L_q(\Omega_{R/2}(x_0))}
$$
for any $z\in \Omega_{3r}(\hat{x})$, where $N=N(q)$.
Since $x, y \in \Omega_{3r}(\hat{x})$, from the above inequality it follows that
\begin{equation}
							\label{eq0714_01}
|u(x) - u(y)| \leq |u(x) - \bar{u}| +  |u(y) - \bar{u}| \leq N r^{1-2/q}\|Du\|_{L_q(\Omega_{R/2}(x_0))}.
\end{equation}
This shows \eqref{200303@eq1}.
\item
$r\ge R/32$.
By again the last inequality in  \cite[p. 2372]{MR3809039} there exists a constant $\bar{u}$ such that
$$
|u(z)-\bar{u}|\le N \int_{\Omega_{R/2}(x_0)}\frac{|Du(\xi)|}{|z-\xi|}\,d\xi\le N R^{1-2/q} \|Du\|_{L_q(\Omega_{R/2}(x_0))}
$$
for any $z\in \Omega_{R/4}(x_0)$, where $N=N(q)$.
Then, proceeding as in \eqref{eq0714_01} with $R$ in place of $r$ along with $R^{1-2/q} \leq N r^{1-2/q} = N|x-y|^{1-2/q}$, we again obtain \eqref{200303@eq1}.
\end{enumerate}
The lemma is proved.
\end{proof}

We finish this subsection with a remark that the arguments employed in this subsection are applicable to higher dimensional cases.

\subsection{Solvability}	
In this subsection, we provide some solvability results for the conormal derivative problem.
Recall that we do not impose any regularity assumptions on the coefficients $A_{\alpha\beta}$ of the operator $\cL$.
The following lemma  follows from \cite[Lemma 3.2]{MR3809039} combined with Lemma \ref{210709@lem2}.

\begin{lemma}		\label{200229@lem1}
Let $\Omega$ be a bounded domain in $\bR^2$ satisfying Assumption \ref{A1}.
Then for any $(u,p)\in W^1_2(\Omega)^2\times L_2(\Omega)$ satisfying
\begin{equation}		\label{200229@A1}
\left\{
\begin{aligned}
\operatorname{div}u=g &\quad\text{in }\, \Omega,\\
\cL u+\nabla p=D_\alpha f_\alpha &\quad \text{in }\, \Omega,\\
\cB u+p\nu= \nu^\alpha f_\alpha &\quad \text{on }\, \partial \Omega,
\end{aligned}
\right.
\end{equation}
where $f_\alpha\in L_2(\Omega)^2$ and $g\in L_2(\Omega)$, we have 
$$
\|Du\|_{L_2(\Omega)}+\|p\|_{L_2(\Omega)}\le N  \|f_\alpha\|_{L_2(\Omega)}+N\|g\|_{L_2(\Omega)},
$$
where $N=N(\lambda)$.
Moreover, for any $f_\alpha\in L_2(\Omega)^2$ and $g\in L_2(\Omega)$, there exists a unique $(u,p)\in \widetilde{W}^1_2(\Omega)^2\times L_2(\Omega)$ satisfying \eqref{200229@A1}.
\end{lemma}

The lemma below concerns the reverse H\"older's inequality; see \cite[Lemmas 3.6 and 3.7]{MR3809039} and \cite[Ch.\,V]{MR0717034}.

\begin{lemma}		\label{200229@lem2}
Let $\Omega$ be a bounded domain in $\bR^2$ satisfying  Assumption \ref{A1}.
There exists $q_0=q_0(\lambda)> 2$  such that if $(u,p)\in W^1_2(\Omega)^2\times L_2(\Omega)$ is a weak solution of \eqref{200229@A1}, where $f_\alpha\in L_\infty(\Omega)^2$ and $g\in L_\infty(\Omega)$,
then for any $q\in [2, q_0]$, $x_0\in \bR^2$, and $R\in (0, R_0]$,  we have 
$$
\big(|D\bar{u}|^{q}+|\bar{p}|^{q}\big)^{1/q}_{B_{R/2}(x_0)}\le N\big(|D\bar{u}|^2+|\bar{p}|^2\big)^{1/2}_{B_{R}(x_0)}+N\big(|\bar{f}_\alpha|^{q}+|\bar{g}|^{q}\big)^{1/q}_{B_{R}(x_0)},
$$
where  $D\bar{u}$, $\bar{p}$, $\bar{f}_\alpha$, and $\bar{g}$ are the extensions of $Du$, $p$, $f_\alpha$, and $g$ to $\bR^2$ so that they are zero on $\bR^2\setminus \Omega$.
Here, the constant $N$ depends only on $\lambda$ and $q$.
\end{lemma}

By using Lemma \ref{200229@lem2} and following the proof of \cite[Lemma 4.4]{MR3906316}, we get the  
solvability of the conormal derivative problem in $\widetilde{W}^1_{q}(\Omega)^2\times L_{q}(\Omega)$ when $q$ is close to $2$. 

\begin{theorem}		\label{210714@thm1}
Let $\Omega$ be a bounded domain in $\bR^2$ with $\operatorname{diam}\Omega\le K$ satisfying  Assumption \ref{A1}, and let 
$$
q\in [q_0', q_0],
$$
where $q_0=q_0(\lambda)>2$ is the constant from Lemma \ref{200229@lem2} and $q_0'$ is the conjugate exponent of $q_0$.
Then for any $(u, p)\in W^1_q(\Omega)^2\times L_q(\Omega)$ satisfying \eqref{200229@A1} with  $f_\alpha\in L_{q}(\Omega)^2$ and $g\in L_{q}(\Omega)$, we have 
$$
\|Du\|_{L_{q}(\Omega)}+\|p\|_{L_{q}(\Omega)}\le N\|f_\alpha\|_{L_{q}(\Omega)}+N\|g\|_{L_{q}(\Omega)},
$$
where $N=N(\lambda, R_0, K, q)$.
Moreover,  for any $f_\alpha\in L_{q}(\Omega)^2$ and $g\in L_{q}(\Omega)$, there exists a unique $(u,p)\in \widetilde{W}^1_{q}(\Omega)^2\times L_{q}(\Omega)$ satisfying \eqref{200229@A1}.
\end{theorem}

We finish this subsection with the following remark.

\begin{remark}		\label{210712@rmk1}
By the same reasoning as in the proof of  \cite[Theorem 2.2]{MR3809039},  one can extend the result in Theorem \ref{210714@thm1} to the problem 
$$
\left\{
\begin{aligned}
\operatorname{div}u=g &\quad\text{in }\, \Omega,\\
\cL u+\nabla p=f+D_\alpha f_\alpha &\quad \text{in }\, \Omega,\\
\cB u+p\nu=\nu^\alpha f_\alpha &\quad \text{on }\, \partial \Omega,
\end{aligned}
\right.
$$
where $f\in \tilde{L}_{2q/(2+q)}(\Omega)^2$, provided that $q\in (2, q_0]$.
In this case, we have 
\begin{equation}		\label{210715@eq2}
\|Du\|_{L_{q}(\Omega)}+\|p\|_{L_{q}(\Omega)}\le N\|f\|_{L_{2q/(2+q)}(\Omega)}+N\|f_\alpha\|_{L_{q}(\Omega)}+N\|g\|_{L_{q}(\Omega)},
\end{equation}
where $N=N(\lambda, R_0, K, q)$.
\end{remark}

\section{Approximated Green function}	\label{S4}

Throughout this section, we assume that the hypotheses of Theorem \ref{M1} hold.
Under the hypotheses, we  derive various estimates for  approximated Green functions of the Stokes system.  
Denote
$q_0'=q_0/(q_0-1)$,
where $q_0=q_0(\lambda)>2$ is the constant from Lemma \ref{200229@lem2}.

For $y\in \Omega$ and  $\varepsilon\in (0,R_0]$, we set 
$$
\Phi_{\varepsilon, y}=-\frac{1}{|\Omega_{\varepsilon}(y)|}\chi_{\Omega_{\varepsilon}(y)}+\frac{1}{|\Omega|}\in \tilde{L}_\infty(\Omega),
$$
where $\chi_{\Omega_{\varepsilon}(y)}$ is the characteristic function.
We define {\em{the approximated Green function for $\cL$ in $\Omega$}} by a pair 
$$
(G_{\varepsilon}(\cdot,y), \Pi_{\varepsilon}(\cdot,y))\in \widetilde{W}^{1}_{q_0}(\Omega)^{2\times 2}\times L_{q_0}(\Omega)^{1\times 2}
$$ 
satisfying 
\begin{equation}		\label{200304@eq1a}
\begin{cases}
\operatorname{div}G_{\varepsilon}(\cdot,y)=0 &\quad\text{in }\, \Omega,\\
\cL G_{\varepsilon}(\cdot,y)+\nabla \Pi_{\varepsilon}(\cdot,y)=\Phi_{\varepsilon,y} I&\quad \text{in }\, \Omega,\\
\cB G_{\varepsilon}(\cdot,y)+ \nu\Pi_{\varepsilon}(\cdot,y)=0 &\quad \text{on }\, \partial \Omega.
\end{cases}
\end{equation}
Such a pair exists for all $y\in \Omega$ and $\varepsilon\in (0, R_0]$ because of Remark \ref{210712@rmk1}.
Note that for each $k\in \{1,2\}$ and $\phi\in W^1_{q_0'}(\Omega)^2$, we have 
$$
\int_{\Omega} A_{\alpha\beta} D_\beta G^{\cdot k}_{\varepsilon}(\cdot,y)\cdot D_\alpha \phi\,dx+\int_\Omega \Pi^k_{\varepsilon}(\cdot,y)\operatorname{div}\phi\,dx=\dashint_{\Omega_{\varepsilon}(y)}\phi^k\,dx-\dashint_\Omega \phi^k \,dx,
$$
where $G_{\varepsilon}^{\cdot k}(\cdot,y)$ is the $k$th column of $G_{\varepsilon}(\cdot,y)$.
Moreover,  for any $q\in (2, q_0]$, using \eqref{210715@eq2} and 
$$
\|\Phi_{\varepsilon, y}\|_{L_{2q/(2+q)}(\Omega)}\le N|\Omega_{\varepsilon}(y)|^{(2-q)/(2q)}\le N \varepsilon^{-1+2/q}, 
$$
we see that 
\begin{equation}		\label{200303@A1}
\|DG_{\varepsilon}(\cdot,y)\|_{L_{q}(\Omega)}+\|\Pi_{\varepsilon}(\cdot,y)\|_{L_{q}(\Omega)}\le N  \varepsilon^{-1+2/q},
\end{equation}
where $N=N(\lambda, R_0, K,q)$.

In the lemma below, for $s \in [1,2)$, we establish a local $L_s$-estimate for $DG_{\varepsilon}(\cdot,y)$ and $\Pi_{\varepsilon}(\cdot,y)$. Note that the estimate is independent of $\varepsilon$ and the distance between $x$ and $y$.

\begin{lemma}		\label{210805@lem2}
Let $x,y\in \Omega$ and $R\in (0, R_0]$.
Then for any $\varepsilon\in (0, R/8]$, we have 
$$
\|DG_{\varepsilon}(\cdot,y)\|_{L_{s}(\Omega_{R}(x))}+\|\Pi_{\varepsilon}(\cdot,y)\|_{L_{s}(\Omega_{R}(x))}\le N R^{-1+2/s},
$$
where $1\le s<2$ and $N=N(\lambda, R_0, K,s)$.
\end{lemma}

\begin{proof}
Thanks to H\"older's inequality, it suffices to consider the case for $q_0'\le s<2$.
Denote
$$
(v,\pi)=(G_{\varepsilon}^{\cdot k}(\cdot,y), \Pi^k_{\varepsilon}(\cdot,y)).
$$
Let $f_\alpha\in L_\infty(\Omega)^2$ and $g\in L_\infty(\Omega)$ with compact support in $\Omega_R(x)$.
By Theorem \ref{210714@thm1}, there exists $(u,p)\in \widetilde{W}^1_{q_0}(\Omega)^2\times L_{q_0}(\Omega)$ satisfying
\begin{equation}		\label{200304@eq1}
\left\{
\begin{aligned}
\operatorname{div}u=g &\quad \text{in }\, \Omega,\\
\cL^\star u+\nabla p=D_\alpha f_\alpha &\quad \text{in }\, \Omega,\\
\cB^\star u+p\nu= \nu^\alpha f_\alpha &\quad \text{on }\, \partial \Omega.
\end{aligned}
\right.
\end{equation}
Moreover, for any $q\in [q_0', q_0]$, 
\begin{equation}		\label{230208_eq2}
\|Du\|_{L_q(\Omega)}+\|p\|_{L_q(\Omega)}\le N\|f_\alpha\|_{L_q(\Omega)}+N\|g\|_{L_q(\Omega)},
\end{equation}
where $N=N(\lambda, R_0, K, q)$.
Testing \eqref{200304@eq1a} and \eqref{200304@eq1} by $u$ and $v$, respectively, we obtain
$$
\int_\Omega D_\alpha v\cdot f_\alpha\,dz+\int_\Omega \pi g\,dz=\dashint_{\Omega_{\varepsilon}(y)} u^k\,dz,
$$
which yields (using $\varepsilon\le R/8$)
\begin{equation}		\label{230208_eq3a}
\bigg|\int_{\Omega_R(x)} D_\alpha v\cdot f_\alpha\,dz+\int_{\Omega_R(x)} \pi g\,dz\bigg|\le \|u\|_{L_\infty(\Omega_{R/8}(y))}.
\end{equation}
Hence, by \eqref{210804@eq3a} with $q=s':=s/(s-1)>2$, H\"older's inequality, and Lemma \ref{210709@lem2}, we obtain that 
$$
\begin{aligned}
&\bigg|\int_{\Omega_{R}(x)} D_\alpha v\cdot f_\alpha\,dz+\int_{\Omega_{R}(x)} \pi g\,dz\bigg|\\
&\le N R^{1-2/s'}\|Du\|_{L_{s'}(\Omega_R(y))}+ N R^{1-2/s}\|u\|_{L_{2s/(2-s)}(\Omega_{R/8}(y))}\\
&\le N R^{1-2/s'}\|Du\|_{L_{s'}(\Omega)}+ N R^{1-2/s}\|Du\|_{L_{s}(\Omega)},
\end{aligned}
$$
where $N=N(\lambda, R_0, K, s)$.
Note that since $q_0'\le s<s'\le q_0$, \eqref{230208_eq2} holds for both $q=s$ and $q=s'$.
Thus from the above inequality, we get 
$$
\begin{aligned}
&\bigg|\int_{\Omega_{R}(x)} D_\alpha v\cdot f_\alpha\,dz+\int_{\Omega_{R}(x)} \pi g\,dz\bigg|\\
&\le N R^{1-2/s'}\big(\|f_\alpha\|_{L_{s'}(\Omega_R(x))}+\|g\|_{L_{s'}(\Omega_R(x))}\big)\\
&\quad + N R^{1-2/s}\big(\|f_\alpha\|_{L_{s}(\Omega_R(x))}+\|g\|_{L_{s}(\Omega_R(x))}\big),
\end{aligned}
$$
which implies that 
$$
\begin{aligned}
&\bigg|\int_{\Omega_{R}(x)} D_\alpha v\cdot f_\alpha\,dz+\int_{\Omega_{R}(x)} \pi g\,dz\bigg|\\
&\le N R^{1-2/s'}\big(\|f_\alpha\|_{L_{s'}(\Omega_R(x))}+\|g\|_{L_{s'}(\Omega_R(x))}\big).
\end{aligned}
$$
Since the above inequality holds for all $f_\alpha\in L_\infty(\Omega)^2$ and $g\in L_\infty(\Omega)$ having compact support in $\Omega_R(x)$, by the duality we obtain the desired estimate.
\end{proof}

If norms are measured away from the pole $y$, we also obtain the following estimate for $s \in (2,q_0]$ uniformly in $\varepsilon$.

\begin{lemma}		\label{210804@lem3}
Let $y\in \Omega$ and  $R\in (0, R_0]$.
Then for any $\varepsilon\in (0, R_0]$, we have
$$
\|DG_\varepsilon(\cdot,y)\|_{L_{s}(\Omega\setminus \overline{B_R(y)})}+\|\Pi_{\varepsilon}(\cdot,y)\|_{L_{s}(\Omega\setminus \overline{B_R(y)})}\le NR^{-1+2/s},
$$
where $2<s\le q_0$ and  $N=N(\lambda, R_0, K,s)$.
\end{lemma}

\begin{proof}
Thanks to \eqref{200303@A1}, it suffices to consider the case of $\varepsilon\le R/16$.
As in the proof of Lemma \ref{210805@lem2}, we denote
$$
(v,\pi)=(G_{\varepsilon}^{\cdot k}(\cdot,y), \Pi^k_{\varepsilon}(\cdot,y)).
$$
Let $f_\alpha \in L_\infty(\Omega)^2$ and $g\in L_\infty(\Omega)$ with compact support in $\Omega\setminus \overline{B_R(y)}$.
By Theorem \ref{210714@thm1}, there exists $(u,p)\in \widetilde{W}^1_{q_0}(\Omega)^2\times L_{q_0}(\Omega)$ satisfying that 
$$
\left\{
\begin{aligned}
\operatorname{div}u=g &\quad \text{in }\, \Omega,\\
\cL^\star u+\nabla p=D_\alpha f_\alpha &\quad \text{in }\, \Omega,\\
\cB^\star u+p\nu=\nu^\alpha f_\alpha &\quad \text{on }\, \partial \Omega,
\end{aligned}
\right.
$$
and that for any $q\in [q_0', q_0]$, 
\begin{equation}		\label{230208_eq3}
\|Du\|_{L_q(\Omega)}+\|p\|_{L_q(\Omega)}\le N\|f_\alpha\|_{L_q(\Omega)}+N\|g\|_{L_q(\Omega)},
\end{equation}
where $N=N(\lambda, R_0, K, q)$.
As in \eqref{230208_eq3a},
\begin{equation}		\label{230208_eq3b}
\bigg|\int_{\Omega\setminus \overline{B_R(y)}} D_\alpha v\cdot f_\alpha\,dz+\int_{\Omega\setminus \overline{B_R(y)}} \pi g\,dz\bigg|\le \|u\|_{L_\infty(\Omega_{R/16}(y))}.
\end{equation}
Now, we consider the following two cases:
$$
s\le 4, \quad s>4.
$$
\begin{enumerate}[i.]
\item
$s\le 4$.
Note that 
$$
\frac{2s}{2+s}\le s', \quad s\le \frac{2s'}{2-s'},
$$
where $s':=s/(s-1)$.
Together with \eqref{210804@eq3a} with $q=s>2$, H\"older's inequality, and Lemma \ref{210709@lem2}, we get from \eqref{230208_eq3b} that 
\begin{align}
\nonumber
&\bigg|\int_{\Omega\setminus \overline{B_{R}(y)}} D_\alpha v\cdot f_\alpha\,dz+\int_{\Omega\setminus \overline{B_{R}(y)}} \pi g\,dz\bigg|\\
\nonumber
&\le NR^{1-2/s}\|Du\|_{L_s(\Omega_{R/2}(y))}+NR^{1-2/s'}\|u\|_{L^{2s'/(2-s')}(\Omega_{R/16}(y))}\\
\label{230208_eq3c}
&\le NR^{1-2/s}\|Du\|_{L^s(\Omega_{R/2}(y))}+NR^{1-2/s'}\|Du\|_{L^{s'}(\Omega)},
\end{align}
where $N=N(\lambda, R_0, K, s)$.
Let $\eta$ be an infinitely differentiable function in $\bR^2$ such that 
$$
0\le \eta \le 1, \quad \eta\equiv 1 \, \text{ in }\, B_{R/2}(y), \quad \operatorname{supp}\eta \subset B_R(y), \quad |\nabla \eta|\le CR^{-1}.
$$
Since $f^\alpha\equiv g\equiv 0$ in $\Omega_R(y)$, 
$(\eta u, \eta p)$ satisfies 
\begin{equation}		\label{230208_eq3d}
\begin{cases}
\operatorname{div} (\eta u)=G &\text{in }\, \Omega,\\
\cL^* (\eta u)+\nabla (\eta p)=F+D_\alpha F_\alpha &\text{in }\, \Omega,\\
\cB^* (\eta u)+(\eta p) \nu=\nu^\alpha F_\alpha & \text{on }\, \partial \Omega,\\
\end{cases}
\end{equation}
where 
$$
F=(A^{\beta \alpha})^\top D_\beta u D_\alpha \eta+p \nabla \eta, \quad F_\alpha=(A^{\beta\alpha})^\top u D_\beta \eta, \quad G=\nabla \eta \cdot u.
$$
By \eqref{210715@eq2} with $q=s$ applied to \eqref{230208_eq3d}, H\"older's inequality, and Lemma \ref{210709@lem2}, 
we obtain that 
\begin{align}
\nonumber
&\|Du\|_{L_s(\Omega_{R/2}(y))}+\|p\|_{L_s(\Omega_{R/2}(y))}\\
\nonumber
&\le N \|F\|_{L_{2s/(2+s)}(\Omega)}+N\|F_\alpha\|_{L_s(\Omega)}+N\|G\|_{L_s(\Omega)}\\
\nonumber
&\le N R^{-1} \big(\|Du\|_{L_{2s/(2+s)}(\Omega_R(y))}+\|p\|_{L_{2s/(2+s)}(\Omega_R(y))}\big)+NR^{-1}\|u\|_{L_s(\Omega_R(y))}\\
\nonumber
&\le NR^{2/s-2/s'} \big(\|Du\|_{L_{s'}(\Omega)}+\|p\|_{L_{s'}(\Omega)}\big)+N R^{2/s-2/s'} \|u\|_{L_{2s'/(2-s')}(\Omega)}\\
\label{230208_eq3e}
&\le NR^{2/s-2/s'} \big(\|Du\|_{L_{s'}(\Omega)}+\|p\|_{L_{s'}(\Omega)}\big).
\end{align}
Combining this together with \eqref{230208_eq3c}, we have 
$$
\begin{aligned}
&\bigg|\int_{\Omega\setminus \overline{B_{R}(y)}} D_\alpha v\cdot f_\alpha\,dz+\int_{\Omega\setminus \overline{B_{R}(y)}} \pi g\,dz\bigg|\\
&\le NR^{1-2/s'}\big(\|Du\|_{L_{s'}(\Omega)}+\|p\|_{L_{s'}(\Omega)}\big).
\end{aligned}
$$
Therefore by \eqref{230208_eq3} with $q=s'$ and the fact that $1-2/s'=-1+2/s$, we see that 
\begin{equation}		\label{230208_eq3f}
\begin{aligned}
&\bigg|\int_{\Omega\setminus \overline{B_{R}(y)}} D_\alpha v\cdot f_\alpha\,dz+\int_{\Omega\setminus \overline{B_{R}(y)}} \pi g\,dz\bigg|\\
&\le NR^{-1+2/s}\big(\|f_\alpha\|_{L_{s'}(\Omega\setminus \overline{B_R(y)})}+\|g\|_{L_{s'}(\Omega\setminus \overline{B_R(y)})}\big).
\end{aligned}
\end{equation}
Since the above inequality holds for all $f_\alpha \in L_\infty(\Omega)^2$ and $g\in L_\infty(\Omega)$ having compact support in $\Omega\setminus \overline{B_R(y)}$, by the duality we get the desired estimate.
\item
$s>4$.
In this case, we have 
$$
s'<\frac{2s}{2+s}, \quad \frac{2s'}{2-s'}<s.
$$
Similar to \eqref{230208_eq3c}, using \eqref{210804@eq3a} with $q=2s'/(2-s')>2$, we obtain that  
$$
\begin{aligned}
&\bigg|\int_{\Omega\setminus \overline{B_{R}(y)}} D_\alpha v\cdot f_\alpha\,dz+\int_{\Omega\setminus \overline{B_{R}(y)}} \pi g\,dz\bigg|\\
&\le N R^{2-2/s'}\|Du\|_{L_{2s'/(2-s')}(\Omega_{R/2}(y))}+ N R^{1-2/s'}\|Du\|_{L_{s'}(\Omega)},
\end{aligned}
$$
where by the same calculation used in deriving \eqref{230208_eq3e} with $s$ replaced with $2s'/(2-s')$ (in the case $s \leq 4$ above, the estimate \eqref{210715@eq2} with $q=2s'/(2-s')$ may be unavailable because $2s'/(2-s')$ may be bigger than $q_0$), we have 
$$
\|Du\|_{L_{2s'/(2-s')}(\Omega_{R/2}(y))}\le  N R^{-1}\big(\|Du\|_{L_{s'}(\Omega)}+\|p\|_{L_{s'}(\Omega)}\big).
$$
Combining these together and utilizing \eqref{230208_eq3} with $q=s'$, we derive \eqref{230208_eq3f}, which implies the desired estimate.
\end{enumerate}
The lemma is proved.
\end{proof}

Based on Lemma \ref{210804@lem3}, we have the following uniform weak $L_2$-estimate.

\begin{lemma}		\label{210719@lem1}
Let $y\in \Omega$. 
Then for any $\varepsilon\in (0, R_0]$, we have 
$$
\|DG_{\varepsilon}(\cdot,y)\|_{L_{2, \infty}(\Omega)}+\|\Pi_{\varepsilon}(\cdot,y)\|_{L_{2, \infty}(\Omega)}\le N,
$$
where $N=N(\lambda, R_0, K)$.
\end{lemma}

\begin{proof}
For $y\in \Omega$ and $t>0$, set 
$$
A_t=\{x\in \Omega: |DG_{\varepsilon}(\cdot,y)|>t\}.
$$
If $t\le R_0^{-1}$, then
\begin{equation}		\label{210719@eq2}
t|A_t|^{1/2}\le R_0^{-1}|\Omega|^{1/2}\le N,
\end{equation}
where $N=N(R_0, K)$.
Otherwise, i.e,  if $t>R_0^{-1}$, 
by Lemma \ref{210804@lem3} with $R=t^{-1}<R_0$ and $s=q_0$, we have
$$
|A_t\setminus \overline{B_R(y)}|\le \frac{1}{t^{q_0}}\int_{A_t\setminus \overline{B_R(y)}} |DG_{\varepsilon}(x,y)|^{q_0}\,dx\le N t^{-2},
$$
where $N=N(\lambda, R_0, K)$.
In fact, the constant $N$ may depends on $q_0$ as well, but $q_0$ is determined by $\lambda$.
Hence, using the fact that 
$$
|A_t\cap \overline{B_R(y)}|\le N R^2=N t^{-2},
$$
we obtain
$$
t|A_t|^{1/2}\le N,
$$
which together with  \eqref{210719@eq2} yields  
$$
\|DG_{\varepsilon}(\cdot,y)\|_{L_{2, \infty}(\Omega)}\le N.
$$
Similarly, we have the estimate for $\Pi_{\varepsilon}(\cdot,y)$.
The lemma is proved.
\end{proof}

\section{Proof of Theorem \ref{M1}}	\label{S5}
Throughout the proof,  we denote by $(G_{\varepsilon}, \Pi_{\varepsilon})$ the approximated Green function constructed in Section \ref{S4}.
Fix $q_1\in (1, 2)$.
Observe that for any $y\in \Omega$ and $\varepsilon\in (0, R_0]$, 
$$
\|G_{\varepsilon}(\cdot,y)\|_{\widetilde{W}^1_{q_1}(\Omega)}+\|\Pi_{\varepsilon}(\cdot,y)\|_{L_{q_1}(\Omega)}\le N,
$$
where $N=N(\lambda, R_0, K)$, which follows from Lemmas \ref{210709@lem2} and \ref{210719@lem1}.
Indeed, we use the following to obtain the above inequality from Lemma \ref{210719@lem1}:
\[
\|u\|_{L_{q_1}(\Omega)}^{q_1} = \int_0^\infty q_1 t^{q_1-1} |\{x \in \Omega : |u(x)| > t \}| \, dt
\]
\[
\leq \int_0^1 q_1 t^{q_1-1} |\Omega| \, d t + \|u\|_{L_{2,\infty}(\Omega)}^2 \int_1^\infty q_1 t^{q_1-3} \, dt.
\]
Hence,  by the weak compactness theorem,
there exist a sequence $\{\varepsilon_\rho\}_{\rho=1}^\infty$ tending to zero and a pair 
$$
(G(\cdot,y), \Pi(\cdot,y))\in \widetilde{W}^{1}_{q_1}(\Omega)^{2\times 2}\times L_{q_1}(\Omega)^{2}
$$
such that
\begin{equation}		\label{210814@A1}
G_{\varepsilon_\rho}(\cdot,y) \rightharpoonup G(\cdot,y) \quad \text{weakly in $\widetilde{W}^{1}_{q_1}(\Omega)^{2\times 2}$},\end{equation}
$$
\Pi_{\varepsilon_\rho}(\cdot,y) \rightharpoonup \Pi(\cdot,y) \quad \text{weakly in $L_{q_1}(\Omega)^{1\times 2}$},
$$
and
\begin{equation}		\label{210809@eq3}
\|G(\cdot,y)\|_{\widetilde{W}^1_{q_1}(\Omega)}+\|\Pi(\cdot,y)\|_{L_{q_1}(\Omega)}\le N.
\end{equation}
Then the pair $(G, \Pi)$ satisfies the properties $(i)$--$(iii)$ in Definition \ref{D1} so that it is a unique Green function of $\cL$ in $\Omega$.
Indeed, the property $(i)$ follows immediately from \eqref{210809@eq3}, and the property $(ii)$ follows by taking $\rho \to \infty$ in the system \eqref{200304@eq1a} with $\varepsilon_\rho$ in place of $\varepsilon$.
To verify the property $(iii)$, we test  \eqref{171010@eq1} and \eqref{200304@eq1a} with $G_{\varepsilon_\rho}^{\cdot k}(\cdot,y)$ and $u$, respectively, to get 
$$
\dashint_{\Omega_{\varepsilon_\rho}(y)}u^k\,dx=-\int_\Omega G_{\varepsilon_\rho}^{\cdot k}(\cdot,y)\cdot f\,dx+\int_\Omega D_\alpha G_{\varepsilon_\rho}^{\cdot k}(\cdot,y)\cdot f_\alpha\,dx+\int_\Omega \Pi^k_{\varepsilon_\rho}(\cdot,y) g\,dx.
$$
This implies that
$$
u^k(y)=-\int_\Omega G^{\cdot k}(\cdot,y)\cdot f\,dx+\int_\Omega D_\alpha G^{\cdot k}(\cdot,y)\cdot f_\alpha\,dx+\int_\Omega \Pi^k(\cdot,y) g\,dx,
$$
provided that $y$ is in the Lebesgue set of $u^k$.
Hence,  $(G, \Pi)$ satisfies the property $(iii)$,  and thus it is the Green function of $\cL$ in $\Omega$.

Now we prove the estimates \eqref{230208_eq1}--\eqref{210721@eq2a} in Theorem \ref{M1}.
Let $x,y\in \Omega$ and $R\in (0, R_0]$.
It then follows from Lemma \ref{210805@lem2} and \eqref{210814@A1} that 
$$
\begin{aligned}
\bigg|\int_{\Omega_R(x)} DG(\cdot, y)\phi\,dz\bigg|
&=\lim_{\rho \to \infty} \bigg|\int_{\Omega_R(x)} DG_{\varepsilon_\rho}(\cdot, y)\phi\,dz\bigg|\\
&\le \limsup_{\rho \to \infty} \|DG_{\varepsilon_\rho}(\cdot, y)\|_{L_{q}(\Omega_R(x))}\|\phi\|_{L_{q'}(\Omega_R(x))}\\
&\le N R^{-1+2/q} \|\phi\|_{L_{q'}(\Omega_R(x))}
\end{aligned}
$$
for  all  $\phi\in L^\infty(\Omega_R(x))$, where $1\le q<2$ and $q'=q/(q-1)$.
Hence, by the duality, 
$$
\|DG(\cdot, y)\|_{L_q(\Omega_R(x))}\le N R^{-1+2/s}.
$$
Similarly we have 
$$
\|\Pi(\cdot, y)\|_{L_q(\Omega_R(x))}\le N R^{-1+2/s}.
$$
This shows \eqref{230208_eq1}.
By the same reasoning, we get  \eqref{210803@eq1} from Lemma \ref{210804@lem3}, which implies \eqref{210721@eq1a} (using \eqref{210804@eq3}), as well as  
 \eqref{210721@eq2} as in the proof of Lemma \ref{210719@lem1}.
Notice that using \eqref{210721@eq1a}, we obtain  
\begin{equation}		\label{210802@eq3}
\Bigg|G(x_0, y)-\dashint_{\Omega_{R/16}(x)} G(z,y)\,dz\Bigg|\le N_0
\end{equation}
for all $x_0, x,y\in \Omega$ and $R\in (0, R_0]$ satisfying $|x-y|\ge R$ and $x_0\in \Omega_{R/16}(x)$,
where $N_0=N_0(\lambda, R_0, K)$.

To prove  \eqref{210721@eq2a}, let $x,y\in \Omega$ with $x\neq y$, and set $|x-y|=\rho$.
We consider the following two cases:
$$
\rho\ge R_0/8, \quad \rho< R_0/8.
$$
\begin{enumerate}[i.]
\item
$\rho\ge R_0/8$.
By \eqref{210809@eq3}, we know that 
\begin{equation}		\label{210814@A2}
\|G(\cdot,y)\|_{L_1(\Omega)}\le N,
\end{equation}
where $N=N(\lambda, R_0, K)$.
Using this and  \eqref{210802@eq3} with $x_0=x$ and $R=R_0/8$, we obtain
$$
|G(x,y)|
\le \Bigg|G(x,y)-\dashint_{\Omega_{R_0/128}(x)}G(z,y)\,dz\Bigg|+\Bigg|\dashint_{\Omega_{R_0/128}(x)}G(z,y)\,dz\Bigg|\le N,
$$
which gives \eqref{210721@eq2a}.
\item
$\rho<R_0/8$.
In this case, by Theorem \ref{thm230119_1} in Appendix,  there is a point $y_0\in \Omega$ with $|y-y_0|\ge R_0/8$ satisfying the following:
there exists a chain of balls   
$B_{\rho_j}(z_j)$,  $j\in \{1, \ldots, k\}$,
where  $z_j\in \Omega$ and  $k\le N\log(R_0/\rho)$, such that 
$$
x\in B_{\rho_1}(z_1), \quad y_0\in B_{\rho_{k}}(z_k)
$$
$$
|z_j-y|\ge 16 \rho_j, \quad \rho_j\le R_0, \quad  \quad j\in \{1,\ldots,k\},
$$
and
$$
\Omega\cap B_{\rho_j}(z_j)\cap B_{\rho_{j+1}}(z_{j+1})\neq \emptyset, \quad j\in \{1,\ldots, k-1\}.
$$
where we may assume that $\rho_k\ge R_0/(8\cdot 17)$.
Set $\Theta_j=\Omega\cap B_{\rho_j}(z_j)$.
For each $j\in \{1, \ldots, k-1\}$, we choose  $\tilde{z}_j\in \Theta_j\cap \Theta_{j+1}$.
It then follows from  \eqref{210802@eq3} that 
$$
\begin{aligned}
\bigg|\dashint_{\Theta_j}G(z,y)\,dz\bigg|
&\le \bigg|\dashint_{\Theta_j}G(z,y)\,dz-G(\tilde{z}_j, y)\bigg|\\
&\quad +\bigg|\dashint_{\Theta_{j+1}}G(z,y)\,dz-G(\tilde{z}_j, y)\bigg|+\bigg|\dashint_{\Theta_{j+1}}G(z,y)\,dz\bigg|\\
&\le 2 N_0+\bigg|\dashint_{\Theta_{j+1}}G(z,y)\,dz\bigg|.
\end{aligned}
$$
By iteration,
$$
\begin{aligned}
\bigg|\dashint_{\Theta_1}G(z,y)\,dz\bigg|
&\le 2 (k-1) N_0+\bigg|\dashint_{\Theta_k} G(z,y)\,dz\bigg|
\\
&\le  2(k-1)N_0+N,
\end{aligned}
$$
where we used \eqref{210814@A2} in the second inequality.
Hence, using  \eqref{210802@eq3} again and the fact that 
$$
k\le N\log\bigg(\frac{K}{|x-y|}\bigg),
$$
we have
$$
\begin{aligned}
|G(x,y)|
&\le \bigg|G(x,y)-\dashint_{\Theta_1}G(z,y)\,dz\bigg|+\bigg|\dashint_{\Theta_1}G(z,y)\,dz\bigg|\\
&\le 2kN_0+N\\
&\le N\log\bigg(\frac{K}{|x-y|}\bigg)+N.
\end{aligned}
$$
\end{enumerate}
We have thus verified \eqref{210721@eq2a}.

Next, we prove the symmetry property \eqref{210720@eq2}.
To this end, we define the Green function $(G^\star, \Pi^\star)$ of the adjoint operator $\cL^\star$ in the same manner that $(G, \Pi)$ is defined for the operator $\cL$.
Let $x,y\in \Omega$ with $x\neq y$, and set $r=|x-y|/2$.
Observe that  $\eta G^\star(\cdot,x)$ and $(1-\eta)G^\star(\cdot,x)$ can be applied to \eqref{210815@A1} as test functions, where  $\eta$ is a smooth function in $\bR^2$ satisfying 
$$
\eta\equiv 0 \, \text{ in }\, B_{r/2}(x), \quad \eta\equiv 1 \, \text{ in }\, \bR^2\setminus B_r(x).
$$
Together with  the continuity of $G^\star(\cdot,x)$ in $\Omega\setminus \{x\}$ and the fact that 
$$
G^\star(\cdot,x)=\eta G^\star(\cdot,x)+(1-\eta)G^\star(\cdot, x),
$$
by testing the $l$th columns of $\eta G^\star(\cdot,x)$ and $(1-\eta) G^\star(\cdot, x)$ to \eqref{210815@A1}, we have  
$$
\int_\Omega A_{\alpha\beta}^{ij}D_\beta G^{jk}(\cdot,y)D_\alpha (G^\star)^{il}(\cdot,x)\,dz=(G^\star)^{kl}(y,x).
$$
Similarly we have 
$$
\int_\Omega A_{\alpha\beta}^{ij}D_\beta G^{jk}(\cdot,y)D_\alpha (G^\star)^{il}(\cdot,x)\,dz=G^{lk}(x,y).
$$
This gives the desired identity \eqref{210720@eq2}.
Finally, using the continuity of $G^\star(\cdot,x)$ and \eqref{210720@eq2}, we see that $G(x,\cdot)$ is continuous in $\Omega\setminus \{x\}$.
Hence, by the continuity of $G(\cdot,y)$ in $\Omega\setminus \{y\}$, we conclude that $G$ is continuous in $\{(x,y)\in \Omega\times \Omega: x\neq y\}$.
This completes the proof of the theorem.
\qed

\section{Appendix}	\label{S6}

Throughout this section, we let $\Omega$ be a bounded domain in $\bR^2$ with $\operatorname{diam}\Omega \leq K$ satisfying Assumption \ref{A1}.
We use abbreviations $B_R=B_R(\mathbf{0})$ and $\Omega_R=\Omega_R(\mathbf{0})$, where $\mathbf{0}=(0,0)$ is the origin in $\bR^2$.
We also denote by $\overline{xy}$ the line segment connecting $x$ and $y$.

\begin{lemma}
							\label{lem0905_1}
Let $\mathbf{0}\in \partial \Omega$, $0<R\le R_0/2$, and $z_1, z_2\in \Omega$ satisfying that 
$$
z_1=(R, 0) \quad \text{and}\quad z_2=(2R, 0)
$$
in the coordinate systems associated with $(\mathbf{0}, R)$ and $(\mathbf{0}, 2R)$, respectively.
Then, 
\[
\overline{z_1z_2} \subset \overline{\Omega_{2R}} \setminus B_R \quad \text{and} \quad R \leq |z_1 - z_2| \leq 1.001 R.
\]
\end{lemma}

\begin{proof}
In the coordinate system associated with $(\mathbf{0},R)$ (call it the first coordinate system), we have $z_1 = (R,0)$ and
\[
\{x: \gamma R < x^1\} \cap B_{R} \subset \Omega_{R} \subset \{x:- \gamma R < x^1\} \cap B_{R}.
\]
Denote
\[
H_R^+ = \{x: \in \bR^2: x^1 > \gamma R\}, \quad H_R^- = \{x: \in \bR^2: x^1 < - \gamma R\}
\]
in the first coordinate system.
Similarly, in the coordinate system associated with $(\mathbf{0}, 2R)$ (call it the second coordinate system), we have $z_2 = (2R,0)$ and
\[
\{x: \gamma (2R) < x^1\} \cap B_{2R} \subset \Omega_{2R} \subset \{x:- \gamma (2R) < x^1\} \cap B_{2R}.
\]
Denote
\[
H_{2R}^+ = \{x: \in \bR^2: x^1 > \gamma (2R)\}, \quad H_{2R}^- = \{x: \in \bR^2: x^1 < - \gamma (2 R) \}
\]
in the second coordinate system.
Note that in the second coordinate system, the coordinates of $z_1$ are not necessarily $(R,0)$.
However, we see that 
\begin{equation}
							\label{eq0906_02}
z_1 = (R\cos \theta, R \sin \theta),
\end{equation}
where $\theta$ is the angle between the first and second coordinate systems with 
\begin{equation}
							\label{eq0906_03}
|\tan\theta| \leq \frac{\gamma\left(2\sqrt{1-\gamma^2}+\sqrt{1-4\gamma^2}\right)}{\sqrt{1-\gamma^2}\sqrt{1-4\gamma^2}-2\gamma^2} < 0.001.
\end{equation}
Indeed, we must have 
\begin{equation}
							\label{eq0906_01}
B_R \cap H_{2R}^+ \subset \bR^2 \setminus H_R^-,
\end{equation}
since otherwise, that is, if there is a point $y$ belonging to $ B_R \cap H_{2R}^+\cap H_R^-$, then by Assumption \ref{A1} with the fact that  $y \in B_R \cap H_{2R}^+$, we have $y \in \Omega$, but from the fact that $y \in B_R\cap H_R^-$, we also have $y \notin \Omega$. This is a contradiction.
Note that $\partial H_R^-$ is a line whose distance from the origin is $\gamma R$.
From this, \eqref{eq0906_01}, and a direct calcuation, we obtain \eqref{eq0906_02} with \eqref{eq0906_03}.
It then follows readily that $\overline{z_1z_2} \subset \overline{\Omega_{2R}} \setminus B_R$ and
$$
\begin{aligned}
|z_1 - z_2| &\leq R\sqrt{5 - 4 \cos \theta_1} \\
&= R \sqrt{5 + 8 \gamma^2 - 4 \sqrt{1- \gamma^2} \sqrt{1-4\gamma^2}} \leq 1.001 R
\end{aligned}
$$
for $\gamma \leq 1/96$, where $\theta_1$ is the largest $\theta$ satisfying \eqref{eq0906_03}. 
\end{proof}

\begin{lemma}
							\label{lem0905_2}
Let $\mathbf{0}\in \partial \Omega$ and $y\in \Omega$ satisfying 
\[
\cR := 
\operatorname{dist}(y,\partial\Omega) = |y| < R_0/4.
\]
Then we have the following.

\begin{enumerate}[$(a)$]
\item If $0<\rho < \cR$, there exist $z_1, z_2 \in \Omega$ such that
$z_1 \in \partial B_\rho(y)$, $z_2 \in \partial B_{2\cR}$, $\overline{z_1z_2} \in \overline{\Omega_{2\cR}} \setminus B_\rho(y)$, and
\[
\cR - \rho \leq |z_1-z_2| \leq \sqrt{5}\cR.
\]
In particular, we have $z_1 = y+(\rho,0)$ and $z_2 = (2\cR,0)$ in the coordinate system associated with $(\mathbf{0}, 2\cR)$.

\item If $\cR \le \rho < R_0/4$, there exist $z_1, z_2 \in \Omega$ such that
$z_1 \in \partial B_\rho(y)$, $z_2 \in \partial B_{4\rho}$, $\overline{z_1z_2} \in \overline{\Omega_{4\rho}} \setminus B_\rho(y)$, and
\[
2\rho \leq |z_1-z_2| \leq \sqrt{17}\rho.
\]
In particular, we have $z_1 = y + (\rho,0)$ and $z_2 = (4\rho,0)$ in the coordinate system associated with $(\mathbf{0}, 4\rho)$.
\end{enumerate}
\end{lemma}

\begin{proof}
We only prove the assertion $(a)$ because the proof of $(b)$ is the same with obvious modifications.
Fix a coordinate system associated with $(\mathbf{0}, 2\cR)$ satisfying 
$$
\{x: \gamma (2\cR) < x^1\} \cap B_{2\cR} \subset \Omega_{2\cR}\subset \{x:- \gamma (2\cR) < x^1\} \cap B_{2\cR}.
$$
Denote $y = (y^1,y^2)$ and observe that 
\begin{equation}
							\label{eq0905_03}
y \in \{x: \gamma (2\cR) < x^1\} \cap B_{2\cR},
\end{equation}
since otherwise, that is, if $y^1 \leq \gamma(2R)$, then
\[
\operatorname{dist}(y,\partial \Omega) \leq 2 \gamma (2R),
\]
which contradicts with the fact that $\operatorname{dist}(y,\partial \Omega) = \cR$ becasue $\gamma \leq 1/96$.
Set 
$$
z_1 = (y^1+\rho,y^2), \quad z_2=(2\cR, 0).
$$
Clearly, $z_1 \in \partial B_\rho(y)$ and $z_1 \in B_{2\cR}$.
From \eqref{eq0905_03}, which means that $y^1+\rho > \gamma(2R) + \rho$, we have
$$
z_1 \in \{x: \gamma (2\cR) < x^1\} \cap B_{2\cR} \subset \Omega_{2\cR}.
$$
Since the first coordinates of the points on the line segment $\overline{z_1z_2}$ are bigger than those of the points in $B_\rho(y)$, it follows that
\[
\overline{z_1z_2} \subset \overline{\Omega_{2\cR}} \setminus B_\rho(y).
\]
Moreover, using the facts that
\[
0 < \gamma(2\cR) + \rho \leq z_1^1 = y^1+\rho \leq \cR + \rho, \quad z_1^2 = y^2 \in [-\cR, \cR],
\]
we obtain
\[
\cR - \rho \leq |z_1 - z_2| \leq \sqrt{5} \cR.
\]
The assertion $(a)$ is proved.
\end{proof}

\begin{theorem}		\label{thm230119_1}
Let $x,y\in \Omega$ with $0<\rho:=|x-y|<R_0/8$.
Then there exists a point $y_0\in \Omega$ with $|y-y_0|\ge R_0/8$ satisfying the following:
there exists a chain of balls   
$B_{\rho_j}(z_j)$,  $j\in \{1, \ldots, k\}$,
where  $z_j\in \Omega$ and  $k\le N\log(R_0/\rho)$, such that 
$$
x\in B_{\rho_1}(z_1), \quad y_0\in B_{\rho_{k}}(z_k)
$$
$$
|z_j-y|\ge 16 \rho_j, \quad \rho_j\le R_0, \quad  \quad j\in \{1,\ldots,k\},
$$
and
$$
\Omega\cap B_{\rho_j}(z_j)\cap B_{\rho_{j+1}}(z_{j+1}) \neq \emptyset, \quad j\in \{1,\ldots, k-1\}.
$$
In the above, $N$ is a universal constant.
\end{theorem}

\begin{proof}
We only present here the detailed proof of the case when $B_{R_0/8}(y)\not\subseteq \Omega$ because the other case is simpler.
Take $\tilde{y}\in \partial \Omega$ such that 
$$
\cR:=\operatorname{dist}(y, \partial \Omega)=|y-\tilde{y}|<R_0/8.
$$
We may assume that $\tilde{y}=\mathbf{0}$ after translating the coordinates.
We consider the following two cases:
$$
\rho<\cR \quad \text{and}\quad \rho\ge \cR.
$$
\begin{enumerate}[(i)]
\item
$\rho<\cR$:
Set
$$
y=(y^1, y^2), \quad w_0=(y^1+\rho, y^2), \quad w_1=(2\cR, 0)
$$
in the coordinate system associated with $(\mathbf{0}, 2\cR)$.
Since $\rho < \cR$, we see that $B_\rho(y) \subset \Omega$.
Denote by $\eta_0$ an arc on $\partial B_\rho(y)$ connecting $x$ and $w_0$.
We then use Lemma \ref{lem0905_2} $(a)$ to find a line segment $\overline{w_0w_1}$ in $\overline{\Omega_{2\cR}} \setminus B_\rho(y)$ and Lemma \ref{lem0905_1} to find a line segment $\overline{w_1w_2}$ in $\overline{\Omega_{4\cR}}\setminus B_{2\cR}$, where
$$
w_2=(4\cR, 0)
$$
in the coordinate system associated with $(\mathbf{0}, 4\cR)$.
We continue until we have $w_m \in \Omega$ whose coordinates are $w_m = (2^m \cR,0)$ in the coordinate system associated with $(\mathbf{0}, 2^m \cR)$, where
$$
R_0/2 \leq 2^m\cR < R_0, \quad m \in \{3,4,\ldots\}.
$$
Denote
\[
\eta_i = \overline{w_{i-1}w_m}, \quad i\in \{1,\ldots, m\}.
\]
Note that 
$$
\ell(\eta_0)\le 2\pi \rho, \quad \cR-\rho\le \ell(\eta_1)\le \sqrt{5}\cR, 
$$
and
$$
2^{i-1}\cR\le \ell(\eta_i)\le 1.001(2^{i-1}\cR), \quad i\in \{2,\ldots, m\},
$$
where $\ell(\eta)$ means the length of a curve $\eta$.
We also note that for $z\in \eta_i$, $i\in \{0,1,\ldots, m\}$, 
$$
|z-y|=\rho \quad \text{if }\, i=0, \quad |z-y|\ge \rho \quad \text{if }\, i=1, 
$$
and 
$$
|z-y|\ge |z|-|y|\ge  (2^{i-1}-1)\cR \quad \text{if }\, i\in \{2,\ldots, m\}.
$$
We now construct the desired chain of balls along the curves $\eta_i$, $i\in \{0,1,\ldots, m\}$, as follows.
\begin{itemize}
\item
$i=0$: In this case, we cover $\eta_0$ with overlapping balls with radius $\rho/16$.
\item
$i=1$.
Let $n$ be the smallest nonnegative integer such that 
$$
\ell(\eta_1)\le 2^n \rho.
$$
Note that 
$$
2^n\le \frac{2\sqrt{5}\cR}{\rho}\le \frac{R_0}{\rho}.
$$
For $z\in \eta_1$ with $|z-w_0|<2\rho$, we cover such points with overlapping balls whose radius is $\rho/16$.
For $z\in \eta_1$ with $|z-w_0|\ge 2\rho$, we find $z_k\in \eta_1$ satisfying $|z_k-w_0|=2^k\rho$, $k\in \{1,\ldots, n-1\}$, and $z_n=w_1$.
Notice that for $z\in \overline{z_k z_{k+1}}$, 
$$
|z-y|\ge |z-w_0|-|w_0-y|=|z-w_0|-\rho\ge |z-w_0|-|z-y|,
$$
which gives 
$$
|z-y|\ge \frac{|z-w_0|}{2}\ge 2^{k-1}\rho.
$$
We then cover $\overline{z_k z_{k+1}}$ with overlapping balls whose radius is $2^{k-1}\rho/16$.
Note that the number of such balls can be bounded by $32$.
\item
$2\le i\le m$: In this case, 
 we cover $\eta_i$ with overlapping balls whose radius is 
$(2^{i-1}-1)\cR/16$.
The number of such balls can be also bounded by $32$.
\end{itemize}
\item
$\rho\ge \cR$:
Set 
$$
y=(y^1, y^2), \quad w_0=(y^1+\rho, y^2), \quad w_1=(4\rho, 0)
$$
in the coordinate system associated with $(\mathbf{0}, 4\rho)$.
Note that
\[
w_0 \in \overline{\Omega_{2\rho}}, \quad B_{(1-8\gamma)\rho}(w_0) \subset \Omega_{4\rho},
\]
and that by Lemma \ref{lem0905_2} $(b)$,  
\[
\overline{w_0 w_1} \subset \overline{\Omega_{4\rho}} \setminus B_\rho(y), \quad 2\rho\le \eta(\overline{w_0 w_1})\le \sqrt{17}\rho.
\]
As in the case $(i)$, we apply Lemma \ref{lem0905_1} with $R=2^{i-1} (4\rho)$ to find $w_i\in \Omega$, $i\in \{2,\ldots, m\}$, such that 
$|w_i|= 2^{i-1}(4\rho) = 2^{i+1}\rho$ and 
\[
2^{i}\rho \leq |w_{i} - w_{i-1}| \leq 1.001 (2^{i}\rho),
\]
where
\[
R_0/2 \leq 2^{m-1}(4\rho) = 2^{m+1}\rho < R_0, \quad m\in \{2,3,\ldots\}.
\]
Set
$$
\eta_i = \overline{w_{i-1}w_i}, \quad i\in \{1,2,\ldots,m\}.
$$
We now construct a curve $\eta_0 \subset \Omega$ connecting a point in $\overline{B_{\rho/16}(x)} \cap \Omega$ and $w_0$ as follows.
Let
$$
x=(x^1, x^2)
$$
in the coordinate system associated with $(\mathbf{0}, 3\rho)$.
In this coordinate system, we have 
\begin{equation}		\label{230120_EQ1}
w_0^1> \gamma(85 \rho)
\end{equation}
since $B_{(1-8\gamma)\rho}(w_0) \subset \Omega$. 
Hence, obviously, 
$$
w_0 \in \{z: \gamma(3\rho) < z^1\} \cap B_{3\rho}\subset \Omega.
$$
If $x^1 > \gamma(3\rho)$, then there is an arc on $\partial B_\rho(y)$ connecting $x$ and $w_0$ inside $\Omega$.
Otherwise, that is, if $x^1 \leq \gamma (3\rho)$, we set
\[
\bar{x} = (x^1+6\gamma\rho,x^2)
\]
so that (using $x^1>-\gamma(3\rho)$)
\begin{equation}
							\label{eq0907_02}
\bar{x} \in \{z: \gamma(3\rho) < z^1\} \cap B_{3\rho} \subset \Omega.
\end{equation}
Moreover, 
$$
\bar{x} \in \overline{B_{\rho/16}(x)}, \quad \rho - 6\gamma \rho \leq \bar{\rho} := |y-\bar{x}| \leq \rho + 6 \gamma \rho.
$$
Notice that 
for $z_0 \in \partial B_{\bar{\rho}}(y)$ satisfying
\[
\operatorname{dist}(w_0,\partial B_{\bar{\rho}}(y)) = |w_0 - z_0| = |\rho - \bar{\rho}| \leq 6 \gamma \rho,
\]
we have $z_0 \in \overline{B_{6\gamma \rho}(w_0)}$, from which together with \eqref{230120_EQ1} we get 
\begin{equation}
							\label{eq0907_03}
z_0 \in \{x: \gamma(3\rho) < x^1\} \cap B_{3\rho}.
\end{equation}
By \eqref{eq0907_02} and \eqref{eq0907_03}, there is an arc $\bar{\eta}_0$ on $\partial B_{\bar{\rho}}(y)$ inside $\Omega$ connecting $\bar{x}$ and $z_0$.
We then set $\eta_0$ to be the union of the arc $\bar{\eta}_0$ and the line segment from $z_0$ to $w_0$.
Notice that $|z-y|\ge \rho-6\gamma \rho$ for $z\in \eta_0$.

From now on we mean by $\eta_0$ the curve which is either the arc inside $\Omega$ from $x$ to $w_0$ on $\partial B_\rho(y)$ (when $x^1 > 3(\gamma\rho)$) or the curve inside $\Omega$ from $\bar{x}$ to $w_0$ (when $x^1 \leq 3(\gamma \rho)$).
Then by the same reasoning as in the case $(i)$, we  construct a chain of balls along the curves $\eta_i$, $i\in \{0,1,\ldots,m\}$ satisfying the required properties in the theorem.
In particular, when $\bar{x}$ is considered, we add $B_{\rho/16}(x)$ and $B_{\rho/16}(\bar{x})$ to the chain.
In this case, we do not concern whether there is a curve inside $\Omega$ connecting $x$ and $\bar{x}$.
\end{enumerate}
Clearly, in both cases, one can set $y_0 := z_k$. The theorem is proved.
\end{proof}

\bibliographystyle{plain}

\end{document}